\documentclass[reqno,11pt]{amsart}
\usepackage{amsmath, amsthm, amssymb, booktabs}
\usepackage{mathrsfs}
\usepackage{color}

\advance \topmargin by -\headheight
\advance \topmargin by -\headsep
     
\evensidemargin \oddsidemargin
\newtheorem{theorem}{Theorem}[section]
\newtheorem{lemma}[theorem]{Lemma}
\newtheorem{corollary}[theorem]{Corollary}

\theoremstyle{definition}

\theoremstyle{remark}

\numberwithin{equation}{section}

\newcommand{\mmod}[1]{\,\,(\text{mod}\,\,#1)}

\def\bfk{{\mathbf k}}

\def\dbN{{\mathbb N}}
\def\dbR{{\mathbb R}}
\def\dbZ{{\mathbb Z}}

\def\grJ{{\mathfrak J}}
\def\grk{{\mathfrak k}}

\def\grm{{\mathfrak m}}\def\grM{{\mathfrak M}}\def\grN{{\mathfrak N}}
\def\grS{{\mathfrak S}}

\def\grK{{\mathfrak K}}

\def\alp{{\alpha}} 
\def\bet{{\beta}}  
 \def\Gam{{\Gamma}}
\def\del{{\delta}} \def\Del{{\Delta}}

\def\tet{{\theta}}  \def\Tet{{\Theta}}
\def\kap{{\kappa}}
\def\lam{{\lambda}}

\def\Ups{{\Upsilon}} 

\def\ome{{\omega}} \def\Ome{{\Omega}} 
\def\d{{\partial}}
\def\eps{\varepsilon}

\def\le{\leqslant} \def\ge{\geqslant}

\def\d{{\,{\rm d}}}

\makeatletter
\@namedef{subjclassname@2020}{\textup{2020} Mathematics Subject Classification}
\makeatother

\begin{document}
\title[Waring's problem]{On Waring's problem:\\ beyond Fre\u \i man's theorem}
\author[J\"org Br\"udern]{J\"org Br\"udern}
\address{Mathematisches Institut, Bunsenstrasse 3--5, D-37073 G\"ottingen, Germany}
\email{jbruede@gwdg.de}
\author[Trevor D. Wooley]{Trevor D. Wooley}
\address{Department of Mathematics, Purdue University, 150 N. University Street, West 
Lafayette, IN 47907-2067, USA}
\email{twooley@purdue.edu}
\subjclass[2020]{11P05, 11P55}
\keywords{Waring's problem, Fre\u \i man's theorem, Hardy-Littlewood method.}
\thanks{First author supported by Deutsche Forschungsgemeinschaft Project Number 
255083470. Second author supported by NSF grants DMS-1854398 and DMS-2001549.}
\date{}

\begin{abstract} Let $k_i\in \dbN$ $(i\ge 1)$ satisfy $2\le k_1\le k_2\le \ldots $. 
Fre\u \i man's theorem shows that when $j\in \dbN$, there exists $s=s(j)\in \dbN$ such 
that all large integers $n$ are represented in the form 
$n=x_1^{k_j}+x_2^{k_{j+1}}+\ldots +x_s^{k_{j+s-1}}$, with $x_i\in \dbN$, if and only 
if $\sum k_i^{-1}$ diverges. We make this theorem effective by showing that, for each 
fixed $j$, it suffices to impose the condition
\[
\sum_{i=j}^\infty k_i^{-1}\ge 2\log k_j +4.71.
\]
More is established when the sequence of exponents forms an arithmetic progression. Thus, 
for example, when $k\in \dbN$ and $s\ge 100(k+1)^2$, all large integers $n$ are 
represented in the form $n=x_1^k+x_2^{k+1}+\ldots +x_s^{k+s-1}$, with $x_i\in \dbN$. 
\end{abstract}
\maketitle

\section{Introduction} Recent advances in the smooth number technology associated with 
Waring's problem (see \cite{BW2022}) make possible an investigation of the cognate 
problem to which Fre\u \i man's theorem provides a qualitative answer. Consider then 
natural numbers $k_i$ $(i\ge 1)$ satisfying $2\le k_1\le k_2\le \ldots $. We address the 
problem of determining circumstances in which, given $j\in \dbN$, there exists a natural 
number $s=s(j)$ such that all large integers $n$ are represented in the form
\[
x_1^{k_j}+x_2^{k_{j+1}}+\ldots +x_s^{k_{j+s-1}}=n,
\]
with $x_i\in \dbN$ $(1\le i\le s)$. Fre\u \i man's theorem, announced in 1949 (see 
\cite{Fre1949}), asserts that such holds if and only if the infinite series $\sum k_i^{-1}$ 
diverges. A formal proof of this conclusion was given by Scourfield in 1960 
(see \cite[Theorem 1]{Sco1960}). We now provide an effective version of this conclusion.

\begin{theorem}\label{theorem1.1}
Let $k_i\in \dbN$ $(i\ge 1)$ satisfy $2\le k_1\le k_2\le \ldots $. Suppose that $s$ is a 
natural number for which
\[
\sum_{i=3}^s\frac{1}{k_i}>2\log k_1+\frac{1}{k_2}+3.20032.
\]
Then all sufficiently large natural numbers $n$ are represented in the form 
\[
x_1^{k_1}+x_2^{k_2}+\ldots +x_s^{k_s}=n,
\]
with $x_i\in \dbN$ $(1\le i\le s)$.
\end{theorem}

Since the hypotheses of Theorem \ref{theorem1.1} impose the condition 
$k_2\ge k_1\ge 2$, one obtains an immediate consequence of this theorem that implies 
Fre\u \i man's theorem.

\begin{corollary}\label{corollary1.2}
Let $k_i\in \dbN$ $(i\ge 1)$ satisfy $2\le k_1\le k_2\le \ldots $, and suppose that 
$j\in \dbN$. Then whenever $s$ is a natural number for which
\begin{equation}\label{1.1}
\sum_{i=j}^{j+s-1}\frac{1}{k_i}\ge 2\log k_j+4.71,
\end{equation}
all sufficiently large natural numbers $n$ are represented in the form 
\[
x_1^{k_j}+x_2^{k_{j+1}}+\ldots +x_s^{k_{j+s-1}}=n,
\]
with $x_i\in \dbN$ $(1\le i\le s)$.
\end{corollary}

By making better use of sharper Weyl exponents available for smaller exponents, most 
particularly in the situation in which one or more of the $k_i$ are equal to $2$, it would not 
be difficult to reduce the number $4.71$ occurring in the lower bound \eqref{1.1} of the 
statement of Corollary \ref{corollary1.2}. Back of the envelope computations suggest that a 
number comfortably below $3.5$ should be accessible. For larger values of $k_1$, and 
$k_2$ large compared to $k_1$, on the other hand, the conclusion of Theorem 
\ref{theorem1.1} has strength reflecting the limits of current technology. Standard 
heuristics from the circle method, meanwhile, suggest that the conclusion of Corollary 
\ref{corollary1.2} should remain valid provided only that
\[
\sum_{i=j}^{j+s-1}\frac{1}{k_i}>4.
\]
If one is prepared to accept a local solubility condition, then the assumption of 
square-root cancellation for the mean values of exponential sums encountered in the 
application of the circle method would reduce the lower bound $4$ here to $2$, while the 
most optimistic heuristics would reduce this number further to $1$.\par   

We now turn to the special case of this variant of Waring's problem involving mixed powers 
in which the exponents consist of consecutive terms of an arithmetic progression. Thus, 
when $k$ and $r$ are non-negative integers with $k\ge 2$, we consider the representation 
of large positive integers $n$ in the shape
\begin{equation}\label{1.2}
x_1^k+x_2^{k+r}+\ldots +x_s^{k+r(s-1)}=n,
\end{equation}
with $x_i\in \dbN$ $(1\le i\le s)$. We denote by $R(k,r)$ the least number $s$ having the 
property that all large integers $n$ are represented in the form \eqref{1.2}. In particular, 
the important number $G(k)$ familiar to aficionados of Waring's problem is equal to 
$R(k,0)$. Moreover, the pioneering work of Roth \cite[Theorem 2]{Rot1951} shows that 
$R(2,1)\le 50$, which is to say that all large enough integers $n$ have a representation in 
the shape
\[
n=x_1^2+x_2^3+\ldots +x_{50}^{51},
\]
with $x_i\in \dbN$ $(1\le i\le 50)$.

\begin{theorem}\label{theorem1.3}
Let $k$ and $r$ be natural numbers with $k\ge 2$. Then, uniformly in $k$ and $r$ one has 
$R(k,r)\le A(r)(k+1)^{r+1}$, where $A(r)=r^{-1}25^r(r+1)^{r+1}$. Meanwhile, when 
$r\ge k$ one has $R(k,r)\le (6k+6)^{2r}$.
\end{theorem} 

It would appear that the only previous work concerning this problem of such generality 
hitherto available in the literature is that due to Scourfield \cite[Theorem 2]{Sco1960}. The 
latter work shows that when $k\ge 12$, one has
\[
R(k,r)\le C(r)k^{4r+1}(\log k)^{2r},
\]
in which $C(r)$ is a quantity depending at most on $r$, but apparently growing somewhat 
more rapidly than $\exp (8r^2)$. Meanwhile, the early work of Roth \cite{Rot1951} 
showing that $R(2,1)\le 50$ has been improved by a sequence of authors over the past 
seven decades (see \cite{Bru1987, Bru1988, For1995, For1996, Tha1968, Tha1980, 
Tha1982, Tha1984, Vau1970, Vau1971}). Most recently, Liu and Zhao \cite{LZ2021} have 
shown that $R(2,1)\le 13$. As $r$ increases in equation \eqref{1.2}, the number of 
summands required to apply available technology increases rapidly. Thus, recent work of 
Kuan, Lesesvre and Xiao \cite[Theorem 2]{KLX2020} asserts that $R(2,2)\le 133$.\par

We isolate two cases of the representation problem \eqref{1.2} for special attention. First, 
in the case $r=1$, we note that Ford \cite[Theorems 2 and 3]{For1996} has shown that 
$R(3,1)\le 72$, and that for large values of $k$ one has $R(k,1)\ll k^2\log k$. A corollary of 
Theorem \ref{theorem1.3} improves the order of magnitude of the latter bound.

\begin{corollary}\label{corollary1.4}
When $k$ is an integer with $k\ge 2$, one has $R(k,1)\le 100(k+1)^2$.
\end{corollary}

Thus, when $k\ge 2$ and $s\ge 100(k+1)^2$, all large integers $n$ possess a 
representation in the shape
\[
x_1^k+x_2^{k+1}+\ldots +x_s^{k+s-1}=n,
\]
with $x_i\in \dbN$ $(1\le i\le s)$. The cognate problem in which one seeks representations 
of large integers $n$ in the shape
\[
x_1^k+x_2^{2k}+\ldots +x_s^{sk}=n,
\]
with $x_i\in \dbN$ $(1\le i\le s)$, is considerably more difficult. Here, by taking $r=k$ in 
Theorem \ref{theorem1.3} we obtain the following conclusion.

\begin{corollary}\label{corollary1.5}
Let $k$ be an integer with $k\ge 2$. Then $R(k,k)\le (6k+6)^{2k}$.
\end{corollary}

For comparison, the aforementioned work of Scourfield \cite{Sco1960} would deliver a 
much weaker bound of the general shape $R(k,k)\ll \exp (ck^2)$ for a suitable $c>0$. It is 
worth remarking, however, that the heuristic arguments noted in the discussion following the 
statement of Corollary \ref{corollary1.2} suggest that one should have bounds of the shape 
$R(k,1)\ll k$ and $R(k,k)\ll e^k$.\par

Our proofs of Theorems \ref{theorem1.1} and \ref{theorem1.3} are based on applications 
of the Hardy-Littlewood method, and the basic infrastructure associated with this treatment 
is outlined in \S2. Then, in \S3 we prepare a novel Weyl-type estimate for exponential sums 
over smooth numbers. This eases our path in subsequent discussions and will likely be of 
independent interest. We combine this estimate with an upper bound for mean values of 
smooth Weyl sums in \S4, making use of our recent work \cite{BW2022} concerning 
Waring's problem. Thereby, we obtain an acceptable upper bound for appropriate sets of 
minor arcs relevant to Theorem \ref{theorem1.1} and the second conclusion of Theorem 
\ref{theorem1.3}. A refinement of this approach in \S5 applies for the minor arc 
contribution needed for the proof of the first conclusion of Theorem \ref{theorem1.3}. The 
corresponding major arc contributions are discussed in \S6, the positivity of the singular 
series requiring some additional discussion in \S7.\par

In this paper the letter $p$ is reserved to denote a prime number. We use the standard 
notation $p^h\|n$ to indicate that $p^h|n$ and $p^{h+1}\nmid n$. Also, we write 
$\|\tet\|$ for $\min\{|\tet-n|:n\in \dbZ\}$ and $e(z)$ for $e^{2\pi {\rm i}z}$. 

\section{Preliminary infrastructure} The proofs of Theorems \ref{theorem1.1} and 
\ref{theorem1.3} make use of the Hardy-Littlewood method, with smooth Weyl sums 
playing a pivotal role. We denote the set of $R$-smooth integers not exceeding $P$ by 
$\mathscr A(P,R)$, so that
\[
\mathscr A(P,R)=\{ n\in [1,P]\cap \dbZ:\text{$p|n$ implies $p\le R$}\}.
\]
We note that the standard theory of smooth numbers shows that whenever $\eta\in (0,1)$, 
then there is a positive number $c_\eta$ with the property that 
$\text{card}(\mathscr A(P,R))\sim c_\eta P$ as $P\rightarrow \infty$ (see for example 
\cite[Lemma 12.1]{Vau1997}).\par

Fix $k_i\in \dbN$ $(i\ge 1)$ with $2\le k_1\le k_2\le \ldots $. Let $s$ be a natural number, 
and define $\tet=\tet_s(\bfk)$ by putting
\begin{equation}\label{2.1}
\tet_s(\bfk)=\sum_{i=1}^s\frac{1}{k_i}.
\end{equation}
For now, it suffices to remark that we have in mind imposing the condition $\tet>2$, 
although we shall later impose more onerous conditions on $s$. We consider a natural 
number $n$ sufficiently large in terms of $s$ and $k_1,\ldots ,k_s$, and we seek a 
representation of $n$ in the form
\begin{equation}\label{2.2}
x_1^{k_1}+x_2^{k_2}+\ldots +x_s^{k_s}=n.
\end{equation}
When $1\le i\le s$, we put
\begin{equation}\label{2.3}
P_i=n^{1/k_i}
\end{equation}
and observe that all positive integral solutions of the Diophantine equation \eqref{2.2} 
satisfy the bound $x_i\le P_i$ $(1\le i\le s)$. Fix $\eta$ to be a positive number sufficiently 
small in terms of $s$ and $k_1,\ldots ,k_s$, in a manner that will become clear in due 
course. Our goal is to establish a lower bound for the number $T(n;\eta)$ of solutions of 
the equation \eqref{2.2} with $x_i\in \mathscr A(P_i,P_i^\eta)$ $(1\le i\le s)$.

\par The smooth Weyl sums 
$f_i(\alp)$ that are key to our arguments are defined by
\begin{equation}\label{2.4}
f_i(\alp)=\sum_{x\in \mathscr A(P_i,P_i^\eta)}e(\alp x^{k_i}).
\end{equation}
Writing
\begin{equation}\label{2.5}
\mathscr F(\alp)=f_1(\alp)f_2(\alp)\cdots f_s(\alp),
\end{equation}
it follows via orthogonality that
\begin{equation}\label{2.6}
T(n;\eta)=\int_0^1 \mathscr F(\alp) e(-n\alp)\d\alp .
\end{equation}

\par We derive an asymptotic formula for $T(n;\eta)$ by means of the circle method, the 
successful application of which requires the introduction of a Hardy-Littlewood dissection. 
Write $L=\log n$. We take the set of major arcs $\grK$ to be the union of the intervals
\[
\grK(q,a)=\{ \alp\in [0,1): |\alp-a/q|\le L^{1/15}n^{-1}\},
\]
with $0\le a\le q\le L^{1/15}$ and $(a,q)=1$. The set of minor arcs complementary to 
$\grK$ is then $\grk=[0,1)\setminus \grK$. Our first objective, which we complete in 
\S\S4 and 5, is to establish that for a suitable positive number $\del$, provided that $s$ is 
suitably large in terms of $\bfk$, one has an upper bound of the shape
\begin{equation}\label{2.7}
\int_\grk |\mathscr F(\alp)|\d\alp \ll n^{\tet-1}L^{-\del}.
\end{equation}
The major arc asymptotics is then the central theme of \S\S6 and 7, where we confirm the 
lower bound
\begin{equation}\label{2.8}
\int_\grK \mathscr F(\alp)e(-n\alp)\d\alp \gg n^{\tet-1},
\end{equation}
again for suitably large values of $s$, and with $n$ sufficiently large in terms of $s$, $\bfk$ 
and $\eta$. By combining the bounds \eqref{2.7} and \eqref{2.8} within \eqref{2.6}, we 
conclude that $T(n;\eta)\gg n^{\tet-1}$, so that all large enough integers $n$ 
possess a representation in the shape \eqref{2.2}. This confirms the respective conclusions 
of Theorems \ref{theorem1.1} and \ref{theorem1.3}, the only problem remaining being 
that of determining how large $s$ must be so that the estimates \eqref{2.7} and 
\eqref{2.8} hold true. Appropriate bounds on $s$ will be determined in \S\S 4, 5 and 7.

\section{An estimate of Weyl-type} This section concerns estimates for the exponential 
sums $f_i(\alp)$ of use on sets of minor arcs more general than the arcs $\grk$ introduced 
in the previous section. Consider a natural number $k\ge 2$ and a large positive number 
$P$. We take $Q$ to be a parameter with $1\le Q\le P^{k/2}$. The major arcs $\grM(Q)$ 
are then defined to be the union of the sets
\[
\grM(q,a;Q)=\{ \alp \in [0,1): |q\alp -a|\le QP^{-k}\},
\]
with $0\le a\le q\le Q$ and $(a,q)=1$. The complementary set of minor arcs is then defined 
by putting $\grm(Q)=[0,1)\setminus \grM(Q)$. Finally, we make use of the dyadically 
truncated set of arcs $\grN(Q)=\grM(Q)\setminus \grM(Q/2)$. Notice that, as a 
consequence of Dirichlet's approximation theorem, one has $[0,1)=\grM(P^{k/2})$.\par

Our interest lies in estimates for the exponential sum
\[
f(\alp;P,R)=\sum_{x\in \mathscr A(P,R)}e(\alp x^k),
\]
valid when $R$ is a positive number with $R\le P^\eta$ for a suitably small positive number 
$\eta$, and $\alp \in \grN(Q)$. In order to describe these estimates, we recall the concept 
of an admissible exponent from the theory of smooth Weyl sums. A real number $\Delta_s$ 
is referred to as an {\it admissible exponent} (for $k$) if it has the property that, whenever 
$\eps>0$ and $\eta$ is a positive number sufficiently small in terms of $\eps$, $k$ and $s$, 
then whenever $2\le R\le P^\eta$ and $P$ is sufficiently large, one has
\[
\int_0^1 |f(\alp;P,R)|^s\d\alp \ll P^{s-k+\Del_s+\eps}.
\]
Here, the underlying parameter is $P$ and the constant implicit in Vinogradov's notation 
may depend on $\eps$, $\eta$, $k$ and $s$. One may confirm that for all positive numbers 
$s$, there is no loss of generality in supposing that one has 
$\max \{ 0,k-s/2\}\le \Del_s\le k$.\par

In order to facilitate concision, from this point onwards we adopt the extended $\eps$, $R$
notation routinely employed by scholars working with smooth Weyl sums while applying the 
Hardy-Littlewood method. Thus, whenever a statement involves the letter $\eps$, then it is 
asserted that the statement holds for any positive real number assigned to $\eps$. Implicit 
constants stemming from Vinogradov or Landau symbols may depend on $\eps$, as well as 
ambient parameters implicitly fixed such as $k$ and $s$. If a statement also involves the 
letter $R$, either implicitly or explicitly, then it is asserted that for any $\eps>0$ there is a 
number $\eta>0$ such that the statement holds uniformly for $2\le R\le P^\eta$. Our 
arguments will involve only a finite number of statements, and consequently we may pass 
to the smallest of the numbers $\eta$ that arise in this way, and then have all estimates in 
force with the same positive number $\eta$. Notice that $\eta$ may be assumed 
sufficiently small in terms of $k$, $s$ and $\eps$.\par

Associated with a family $(\Del_s)_{s>0}$ of admissible exponents for $k$ is the number
\begin{equation}\label{3.1}
\tau(k)=\max_{w\in \dbN}\frac{k-2\Del_{2w}}{4w^2},
\end{equation}
an exponent which satisfies the bound $\tau(k)\le 1/(4k)$. For each positive number $s$, 
one then has the related number
\begin{equation}\label{3.2}
\Del_s^*=\min_{0\le t\le s-2}\left(\Del_{s-t}-t\tau(k)\right),
\end{equation}
which we have described elsewhere as an {\it admissible exponent for minor arcs} (see the 
preamble to \cite[Theorem 5.2]{BW2022} for a discussion of these exponents).\par

We recall two consequences of our recent work \cite{BW2022} on Waring's problem.

\begin{lemma}\label{lemma3.1} Suppose that $k\ge 2$, $s\ge 2$ and $\Del_s^*$ is an 
admissible exponent for minor arcs satisfying $\Del_s^*<0$. Let $\kap$ be a positive 
number with $\kap\le k/2$. Then, whenever $P^\kap\le Q\le P^{k/2}$, one has the bound
\[
\int_{\grm(Q)}|f(\alp;P,R)|^s\d\alp \ll_\kap P^{s-k}Q^{\eps-2|\Del_s^*|/k}.
\]
\end{lemma}

\begin{proof} This is immediate from \cite[Theorem 5.3]{BW2022}.
\end{proof}

\begin{lemma}\label{lemma3.2} Suppose that $s$ is a real number with $s\ge 4k$ and 
$\Del_s$ is an admissible exponent. Then whenever $Q$ is a real number with 
$1\le Q\le P^{k/2}$, one has the uniform bound
\[
\int_{\grM(Q)}|f(\alp;P,R)|^s\d\alp \ll P^{s-k}Q^{\eps+2\Del_s/k}.
\]
\end{lemma}

\begin{proof} For the sake of concision, write $f(\alp)$ for $f(\alp;P,R)$. Suppose first that 
$P^{1/(2k)}<Q\le P^{k/2}$. Then the conclusion of \cite[Theorem 4.2]{BW2022} shows 
that whenever $s\ge 2$, one has
\[
\int_{\grM(Q)}|f(\alp)|^s\d\alp \ll P^{s-k+\eps}Q^{2\Del_s/k},
\]
and the desired conclusion is immediate.\par

In order to handle the range of $Q$ with $1\le Q\le P^{1/(2k)}$, we turn to the bounds 
made available in \cite{VW1991}. We take a pedestrian approach sufficient for our 
subsequent application, though we note that with greater effort the condition $s\ge 4k$ 
could be relaxed at this point. Suppose first that $\alp \in \grM(Q)$ for some real number 
$Q$ satisfying 
\[
\exp ((\log P)^{1/3})\le Q\le P^{1/(2k)}.
\]
When $a\in \dbZ$ and $q\in \dbN$ satisfy $(a,q)=1$ and 
$0\le a\le q\le \tfrac{1}{2}P^{k/2}$, the intervals $\grM (q,a;\tfrac{1}{2}P^{k/2})$ 
comprising $\grM(\tfrac{1}{2}P^{k/2})$ are disjoint. For 
$\alp \in \grM(q,a;\frac{1}{2}P^{k/2})\subseteq \grM(\tfrac{1}{2}P^{k/2})$, we put
\[
\Ups(\alp)=(q+P^k|q\alp -a|)^{-1}.
\]
Meanwhile, for $\alp\in [0,1)\setminus \grM(\tfrac{1}{2}P^{k/2})$, we put $\Ups(\alp)=0$. 
Note that when $\alp \in \grN(Q)$ one has $\Ups(\alp)\ll Q^{-1}$. Then as a consequence 
of \cite[Lemma 7.2]{VW1991}, much as in the argument leading to 
\cite[equation (6.3)]{BW2022}, we find that when $\alp \in \grN(Q)$, one has
\[
f(\alp)\ll P(\log P)^3\Ups(\alp)^{-\eps+1/(2k)}+P^{1-1/(2k)}\ll PQ^{2\eps -1/(2k)}.
\]
We remark in this context that the constraint $k\ge 3$ of \cite[equation (6.3)]{BW2022} is 
unnecessary in present circumstances. When instead
\[
1\le Q\le \exp((\log P)^{1/3}),
\]
we appeal to \cite[Lemma 8.5]{VW1991}, deducing as in the cognate argument associated 
with \cite[Theorem 6.1]{BW2022} that
\[
f(\alp)\ll P\Ups(\alp)^{-\eps+1/k}+P\exp(-(\log P)^{1/3})\ll PQ^{-1/(2k)}.
\]
Again, the constraint $k\ge 3$ of \cite{BW2022} is unnecessary in present circumstances. 
Thus, in view of the hypothesis $s\ge 4k$, it follows that when $1\le Q\le P^{1/(2k)}$ one 
has
\begin{align*}
\int_{\grN(Q)}|f(\alp)|^s\d\alp &\ll P^sQ^{\eps-s/(2k)}\text{mes}(\grM(Q))\\
&\ll P^sQ^{\eps -s/(2k)}(Q^2P^{-k})\ll P^{s-k}Q^\eps .
\end{align*}
This again delivers the estimate asserted in the statement of the lemma, since
\[
\int_{\grM(Q)}|f(\alp)|^s\d\alp \le \sum_{\substack{l=0\\ 2^l\le Q}}^\infty 
\int_{\grN(2^{-l}Q)}|f(\alp)|^s\d\alp \ll P^{s-k}Q^\eps .
\]
This completes the proof of the lemma.
\end{proof}

We obtain a pointwise bound for $f(\alp)=f(\alp;P,R)$ when $\alp \in \grm(Q)$ by 
application of the Gallagher-Sobolev inequality.

\begin{lemma}\label{lemma3.3} Suppose that $k\ge 2$ and $0<\rho(k)<2\tau(k)/k$. Then, 
uniformly in $1\le Q\le P^{k/2}$, one has the bound
\[
\sup_{\alp\in \grm(Q)}|f(\alp;P,R)|\ll PQ^{-\rho(k)}.
\]
\end{lemma}

\begin{proof} We consider in the first instance the situation in which 
$P^{1/(2k)}\le Q\le P^{k/2}$. Here, we apply Lemma \ref{lemma3.1} with $s=u+tk$, 
where $u=4^k$ and $t$ is sufficiently large in terms of $k$. The value of $u$ here has 
been chosen large enough that the classical theory of Waring's problem is comfortably 
applicable. With more care one could work with a choice for $u$ little more than $k\log k$. 
On considering the underlying Diophantine equation, working with the value of $u$ already 
chosen, it follows from Hua's lemma and a routine application of the circle method along the 
lines described in \cite[Chapter 2]{Vau1997} that
\begin{equation}\label{3.3}
\int_0^1|f(\alp)|^u\d\alp \le \int_0^1\biggl| \sum_{1\le x\le P}e(\alp x^k)\biggr|^u
\d\alp \ll P^{u-k}.
\end{equation} 
In particular, the exponent $\Del_{s-tk}=0$ is admissible for $k$, and thus it follows from 
\eqref{3.2} that $\Del_s^*=-tk\tau(k)$ is an admissible exponent for minor arcs. We 
therefore infer from Lemma \ref{lemma3.1} that
\begin{equation}\label{3.4}
\int_{\grm(Q)}|f(\alp)|^s\d\alp \ll P^{s-k}Q^{\eps-2t\tau(k)}.
\end{equation}

\par Consider next a real number $\alp$ with $\alp\in \grm(Q)$, and let $\del$ be any real 
number with $|\del|\le P^{-k}$. Suppose, if possible, that $\alp+\del\in \grM(Q/2)$. In 
such circumstances, there exist $a\in \dbZ$ and $q\in \dbN$ with $(a,q)=1$, $1\le q\le Q/2$ 
and $|q(\alp+\del)-a|\le \tfrac{1}{2}QP^{-k}$. Consequently, one has
\[
|q\alp -a|\le q|\del|+\tfrac{1}{2}QP^{-k}\le QP^{-k},
\]
whence $\alp\in \grM(Q)$. This yields a contradiction, so we are forced to conclude that 
$\alp+\del\in \grm(Q/2)$. This observation allows us to estimate $f(\alp)$ pointwise on 
$\grm(Q)$ in terms of mean values for $f(\alp)$ over $\grm(Q/2)$. Indeed, as a 
consequence of the Sobolev-Gallagher inequality (see for example Montgomery 
\cite[Lemma 1.1]{Mon1971}), we have
\[
|f(\alp)|^s\le (2P^{-k})^{-1}\int_{|\bet-\alp|\le P^{-k}}|f(\bet)|^s\d\bet 
+s\int_{|\bet-\alp|\le P^{-k}}|f'(\bet)f(\bet)^{s-1}|\d\bet .
\]
Hence, whenever $\alp \in \grm(Q)$, we infer that
\begin{equation}\label{3.5}
|f(\alp)|^s\ll P^kI_1+I_2,
\end{equation}
where
\[
I_1=\int_{\grm(Q/2)}|f(\bet)|^s\d\bet \quad \text{and}\quad I_2=\int_{\grm(Q/2)}
|f'(\bet)f(\bet)^{s-1}|\d\bet .
\]

\par The bound
\begin{equation}\label{3.6}
I_1\ll P^{s-k}Q^{\eps-2t\tau(k)}
\end{equation}
follows from \eqref{3.4}. Meanwhile, by applying H\"older's inequality, we see that
\begin{equation}\label{3.7}
I_2\le I_3^{1/u}I_4^{1-1/u},
\end{equation}
where
\begin{equation}\label{3.8}
I_3=\int_0^1|f'(\bet)|^u\d\bet \quad \text{and}\quad I_4=\int_{\grm(Q/2)}
|f(\bet)|^v\d\bet ,
\end{equation}
in which
\[
v=\frac{s-1}{1-1/u}.
\]

\par Recall that $u=4^k$ is even. Then since
\[
f'(\bet)=2\pi {\rm i}\sum_{x\in \mathscr A(P,R)}x^ke(\bet x^k),
\]
it follows from \eqref{3.8} by considering the underlying Diophantine equations that
\[
I_3\le (2\pi P^k)^u\int_0^1|f(\bet)|^u\d\bet .
\]
On recalling \eqref{3.3}, therefore, we deduce that
\begin{equation}\label{3.9}
I_3\ll (P^k)^uP^{u-k}.
\end{equation}
Meanwhile, since $s>u$ we have $v>s$, and so it follows from \eqref{3.8} via Lemma 
\ref{lemma3.1} that
\begin{equation}\label{3.10}
I_4\ll P^{v-k}Q^{\eps -2|\Del_v^*|/k},
\end{equation}
where
\[
\Del_v^*=-(v-u)\tau(k)=-\Bigl( \frac{s-u}{1-1/u}\Bigr) \tau(k)=-\frac{tk\tau(k)}{1-1/u}.
\]
On substituting \eqref{3.9} and \eqref{3.10} into \eqref{3.7}, we find that
\begin{equation}\label{3.11}
I_2\ll P^k(P^{u-k})^{1/u}(P^{v-k})^{1-1/u}Q^{\eps-2t\tau(k)}\ll P^sQ^{\eps-2t\tau(k)}.
\end{equation}

\par On substituting \eqref{3.6} and \eqref{3.11} into \eqref{3.5}, we conclude that
\[
|f(\alp)|^s\ll P^sQ^{\eps-2t\tau(k)}.
\]
Thus, whenever $\alp \in \grm(Q)$, we have $f(\alp)\ll PQ^{\eps-\kap}$, where
\[
\kap=\frac{2t\tau(k)}{u+tk}.
\]
We now take $t$ sufficiently large in terms of $k$ and obtain the upper bound
\[
f(\alp)\ll PQ^{\eps-2\tau(k)/k}.
\]
This confirms the upper bound that we sought when $P^{1/(2k)}\le Q\le P^{k/2}$.\par

In order to handle the range of $Q$ with $1\le Q<P^{1/(2k)}$, just as in the proof of 
Lemma \ref{lemma3.2} we turn to the bounds made available in \cite{VW1991}. For the 
sake of concision we adopt the notation of the proof of the latter lemma. Suppose first that 
one has $(\log P)^{60ks}\le Q\le P^{1/(2k)}$ and $\alp \in \grm(Q)$. Then 
\cite[Lemma 7.2]{VW1991} delivers the bound
\begin{align*}
f(\alp)&\ll P(\log P)^3\Ups(\alp)^{-\eps+1/(2k)}+P^{1-\tau(k)+\eps}\\
&\ll PQ^{-1/(3k)}\ll PQ^{\eps-2\tau(k)/k}.
\end{align*}
When instead $1\le Q\le (\log P)^{60ks}$, we appeal to \cite[Lemma 8.5]{VW1991}, 
deducing that
\begin{align*}
f(\alp)&\ll P\Ups(\alp)^{-\eps+1/k}+P(\log P)^{-60ks}\\
&\ll PQ^{-1/(2k)}\ll PQ^{\eps-2\tau(k)/k}.
\end{align*}
Thus, in all circumstances, we have the estimate asserted in the statement of the lemma, 
and the proof is complete.
\end{proof}

For the purposes of this paper, we apply a bound for $\tau(k)$ sufficient for our 
applications, though falling very slightly short of the sharpest bound attainable using current 
technology. In this context, it is useful to introduce the exponent $\ome=\ome(k)$, defined 
by
\begin{equation}\label{3.12}
\ome(k)=1/(Dk^2),\quad \text{where}\quad D=4.51396.
\end{equation}

\begin{lemma}\label{lemma3.4}
When $k\ge 2$ and $1\le Q\le P^{k/2}$, one has the uniform bound 
\[
\sup_{\alp\in \grm(Q)}|f(\alp;P,R)|\ll PQ^{-\ome}.
\]
\end{lemma}

\begin{proof} When $k\ge 4$, it is shown in \cite[Lemma 7.1]{BW2022} that there is a 
family of admissible exponents satisfying the property that $\tau(k)\ge (Ck)^{-1}$, where 
$C=9.027901<2D$. Thus
\[
\frac{2}{k}\tau(k)>\frac{1}{Dk^2},
\]
and the desired conclusion follows from Lemma \ref{lemma3.3}.\par

When $k$ is equal to $2$ or $3$, we appeal to the formula \eqref{3.1} with the crude 
bound on admissible exponents available from Hua's lemma (see 
\cite[Lemma 2.5]{Vau1997}). Thus, we have the admissible exponent $\Del_{2^k}=0$ 
since
\[ 
\int_0^1|f(\alp)|^{2^k}\d\alp \le \int_0^1 \biggl|\sum_{1\le x\le P}e(\alp x^k)
\biggr|^{2^k}\d\alp \ll P^{2^k-k+\eps},
\]
and hence we deduce via \eqref{3.1} that $\tau(2)\ge 1/8$ and $\tau(3)\ge 3/64$. Thus
\[
\frac{2}{2}\tau(2)\ge \frac{1}{8}>\frac{1}{18}>\ome (2)\quad \text{and}\quad 
\frac{2}{3}\tau(3)\ge \frac{1}{32}>\frac{1}{40}>\ome (3).
\]
In each of these cases, the desired conclusion again follows from Lemma \ref{lemma3.3}.
\end{proof}

We finish this section with a formulation of our new minor arc estimate of sufficient flexibility 
that further applications may be anticipated.

\begin{theorem}\label{theorem3.5} Suppose that $a\in \dbZ$ and $q\in \dbN$ satisfy 
$(a,q)=1$. Then one has
\[
f(\alp;P,R)\ll P\left( \lam^{-1}+\lam P^{-k}\right)^\ome,
\]
where $\lam =q+P^k|q\alp -a|$.
\end{theorem}

\begin{proof} We begin by establishing the superficially weaker assertion that, whenever 
$\alp\in \dbR$, $c\in \dbZ$ and $t\in \dbN$ satisfy $(c,t)=1$ and $|\alp -c/t|\le 1/t^2$, then
\begin{equation}\label{3.13}
f(\alp)\ll P(t^{-1}+tP^{-k})^\ome .
\end{equation}
From this assertion, it follows via a standard transference principle (see for example 
\cite[Lemma 14.1]{Woo2015b}) that the conclusion of the lemma holds.\par

Suppose then that $c$ and $t$ satisfy the relations $(c,t)=1$ and $|\alp -c/t|\le 1/t^2$. We 
apply Dirichlet's approximation theorem. Thus, there exist $b\in \dbZ$ and $r\in \dbN$ with 
$(b,r)=1$ satisfying $1\le r\le P^{k/2}$ and $|r\alp-b|\le P^{-k/2}$. We now put
\[
Q=\max \{ r, P^k|r\alp -b|\} \le P^{k/2},
\]
so that $r\le Q$ and $|r\alp -b|\le QP^{-k}$, and either $r\ge Q$ or 
$|r\alp -b|\ge QP^{-k}$. Thus $\alp \in \grM(Q)\setminus \grM(Q/2)\subseteq \grm(Q/2)$, 
and it follows from Lemma \ref{lemma3.4} that
\begin{equation}\label{3.14}
f(\alp)\ll PQ^{-\ome}.
\end{equation}
When $c/t\ne b/r$, it follows from the triangle inequality that
\[
\frac{1}{tr}\le \Bigl| \frac{c}{t}-\frac{b}{r}\Bigr| \le \Bigl| \alp -\frac{c}{t}\Bigr| 
+\Bigl| \alp - \frac{b}{r}\Bigr| \le \frac{1}{t^2}+\frac{Q}{rP^k},
\]
whence
\[
1\le \frac{r}{t}+\frac{tQ}{P^k}.
\]
Thus, we have either $r\ge \tfrac{1}{2}t$ or $Q\ge \tfrac{1}{2}P^k/t$. When instead 
$c/t=b/r$, we have $b=c$ and $r=t$, and the same conclusion holds. In either case, 
therefore, we find that $Q=\max\{ r,P^k|r\alp -b|\}\ge \tfrac{1}{2}\min \{ t,P^k/t\}$. 
Thus, we infer from \eqref{3.14} that
\[
f(\alp)\ll P\left( t^{-\ome}+(P^k/t)^{-\ome}\right) \ll P( t^{-1}+tP^{-k})^\ome.
\]
Thus the desired conclusion \eqref{3.13} follows, and the proof of the theorem is complete.
\end{proof}

\section{The minor arc contribution for ascending powers} We now address the 
representation problem \eqref{2.2} and adopt the notation of \S2. In situations wherein 
$k_2$ may be substantially larger than $k_1$, we apply a Weyl-type estimate only for the 
exponential sum $f_1(\alp)$, estimating the remaining ones in mean. Put
\begin{equation}\label{4.1}
F_1(\alp)=f_1(\alp),\quad G_1(\alp)=f_2(\alp)f_3(\alp)\cdots f_s(\alp),
\end{equation}
and note that in view of \eqref{2.1}, \eqref{2.3} and \eqref{2.4}, one has 
$F_1(0)G_1(0)\asymp n^\tet$. We take $Q$ to be a parameter with 
$1\le Q\le n^{1/2}$ and define a Hardy-Littlewood dissection in accordance with that 
introduced in \S3. Thus, the major arcs $\grM(Q)$ are defined to be the union of the sets
\[
\grM(q,a;Q)=\{ \alp\in [0,1):|q\alp-a|\le Qn^{-1}\},
\]
with $0\le a\le q\le Q$ and $(a,q)=1$, and the associated set of minor arcs are defined by 
setting $\grm(Q)=[0,1)\setminus \grM(Q)$. Also, we put 
$\grN(Q)=\grM(Q)\setminus \grM(Q/2)$. Note that $\grN(Q)\subseteq \grm(Q/2)$. Since
$n=P_i^{k_i}$ for each $i$, these definitions align with those of \S3 when considering the 
smooth Weyl sum $f_i(\alp)$.\par

We begin by recording a Weyl-type estimate for $F_1(\alp)$.

\begin{lemma}\label{lemma4.1}
When $1\le Q\le n^{1/2}$, one has the bound
\[
\sup_{\alp\in \grm(Q)}|F_1(\alp)|\ll F_1(0)Q^{-1/(Dk_1^2)}.
\]
\end{lemma}

\begin{proof} In view of \eqref{4.1}, this estimate is immediate from Lemma 
\ref{lemma3.4}.
\end{proof}

The mean value estimate that we obtain for $G_1(\alp)$ depends on admissible exponent 
bounds. Here we note that, whenever $v$ is even, the corollary to 
\cite[Theorem 2.1]{Woo1993} shows that the exponent $\Del_v$ is admissible for $k\ge 4$, 
where $\Del_v$ is the unique positive solution of the equation
\begin{equation}\label{4.2}
\Del_ve^{\Del_v/k}=ke^{1-v/k}.
\end{equation}
When $k$ is equal to $2$ or $3$, the admissible exponents available from Hua's lemma 
show that the real numbers $\Del_v$ defined via \eqref{4.2} are admissible. Of course, 
much sharper estimates are known in these cases (see \cite{Woo2015a} for the sharpest 
available conclusions when $k=3$). We note that the exponent $\Del_s$ in \cite{Woo1993} 
corresponds to our $\Del_v$ with $v=2s$, owing to the slightly different definition employed 
therein.\par

We next provide an upper bound for the mean value of $|G_1(\alp)|$ over the intermediate 
set of arcs $\grN(Q)$. In this context, it is convenient to introduce the quantity
\begin{equation}\label{4.3}
\Phi_1=\sum_{i=2}^s\frac{1}{k_i}.
\end{equation}

\begin{lemma}\label{lemma4.2}
When $1\le Q\le n^{1/2}$, one has the bound
\[
\int_{\grN(Q)}|G_1(\alp)|\d\alp \ll G_1(0)n^{-1}Q^{2\Tet_1},
\]
where $\Tet_1=e^{1-\Phi_1+2/k_2}$.
\end{lemma}

\begin{proof} Define the exponents
\begin{equation}\label{4.4}
t_i=k_i\Phi_1\quad (2\le i\le s).
\end{equation}
Then it follows from \eqref{4.3} that we have
\[
\sum_{i=2}^s\frac{1}{t_i}=\frac{1}{\Phi_1}\sum_{i=2}^s\frac{1}{k_i}=1,
\]
and hence an application of H\"older's inequality leads us from \eqref{4.1} to the bound
\begin{equation}\label{4.5}
\int_{\grN(Q)}|G_1(\alp)|\d\alp \le \prod_{i=2}^s I_i^{1/t_i},
\end{equation}
where
\begin{equation}\label{4.6}
I_i=\int_{\grN(Q)}|f_i(\alp)|^{t_i}\d\alp .
\end{equation}

\par For each index $i$, the largest even integer not exceeding $t_i$ is larger than $t_i-2$, 
and hence it follows from \eqref{4.2} that there is an exponent $\Del_{t_i}$ admissible for 
$k_i$ with
\[
\Del_{t_i}<k_ie^{1-(t_i-2)/k_i}.
\]
Since $\grN(Q)\subseteq \grM(Q)$, we find from \eqref{4.6} via Lemma \ref{lemma3.2} 
that
\[
I_i\ll P_i^{t_i-k_i}Q^{\eps+2\Del_{t_i}/k_i}.
\]
Thus, in view of \eqref{2.3} and \eqref{4.4}, we have
\begin{equation}\label{4.7}
I_i\ll P_i^{t_i}n^{-1}Q^{2\del_i}\quad (2\le i\le s),
\end{equation}
where
\[
\del_i=e^{1-\Phi_1+2/k_i}.
\]
On substituting \eqref{4.7} into \eqref{4.5}, we conclude that
\begin{equation}\label{4.8}
\int_{\grN(Q)}|G_1(\alp)|\d\alp \ll G_1(0)n^{-1}Q^{2\Ome_1},
\end{equation}
where
\[
\Ome_1=e^{1-\Phi_1+2/k_2}\sum_{i=2}^s\frac{1}{t_i}=e^{1-\Phi_1+2/k_2}=\Tet_1.
\]
The conclusion of the lemma is therefore immediate from \eqref{4.8}.
\end{proof} 

By combining the conclusions of Lemmata \ref{lemma4.1} and \ref{lemma4.2}, we obtain a 
minor arc estimate sufficient for our proof of Theorem \ref{theorem1.1}. Here and 
henceforth, we fix $\eta$ to be a positive number sufficiently small in terms of 
$k_1,k_2,\ldots ,k_s$, and $\eps$, in the context of the estimates of this and the previous 
section relevant for the various admissible exponents encountered. Also, we recall the 
notation of writing $L$ for $\log n$.

\begin{lemma}\label{lemma4.3}
Suppose that
\begin{equation}\label{4.9}
\sum_{i=2}^s\frac{1}{k_i}>2\log k_1+\frac{2}{k_2}+1+\log (2D).
\end{equation}
Then there is a positive number $\del$ having the property that
\[
\int_\grk |\mathscr F(\alp)|\d\alp \ll n^{\tet-1}L^{-\del}.
\]
\end{lemma}

\begin{proof} By referring to the definition of $\grk$ in \S2, we see that $\grk\subset 
\grm(L^{1/15})$. When $L^{1/15}\le Q\le n^{1/2}$, it follows from Lemmata 
\ref{lemma4.1} and \ref{lemma4.2} that
\begin{align*}
\int_{\grN(Q)}|F_1(\alp)G_1(\alp)|\d\alp &\le \Bigl( \sup_{\alp \in \grN(Q)}|F_1(\alp)|\Bigr) 
\int_{\grN(Q)}|G_1(\alp)|\d\alp \\
&\ll F_1(0)Q^{-1/(Dk_1^2)}G_1(0)n^{-1}Q^{2\Tet_1}.
\end{align*}
Provided that the hypothesis \eqref{4.9} holds, it follows from \eqref{4.3} that
\[
e^{\Phi_1}>2e^{1+2/k_2}Dk_1^2,
\]
whence $2\Tet_1<1/(Dk_1^2)$. Put
\[
\del=\frac{1}{15}\Bigl(\frac{1}{Dk_1^2}-2\Tet_1\Bigr).
\]
Then, on recalling \eqref{2.5}, we may conclude thus far that
\[
\int_{\grN(Q)}|\mathscr F(\alp)|\d\alp \ll n^{\tet-1}Q^{-15\del}.
\]
But $\grk$ is covered by the sets $\grN(Q)$ via a dyadic dissection, and we see that
\begin{align*}
\int_\grk |\mathscr F(\alp)|\d\alp &\le
\sum_{\substack{l=0\\ 2^l\le n^{1/2}L^{-1/15}}}^\infty \int_{\grN(2^{-l}n^{1/2})}
|\mathscr F(\alp)|\d\alp \\
&\ll n^{\tet-1}\sum_{\substack{l=0\\ 2^l\le n^{1/2}L^{-1/15}}}^\infty 
(2^{-l}n^{1/2})^{-15\del}\\
&\ll n^{\tet-1}L^{-\del}.
\end{align*}
\end{proof}

We complete this section by addressing the particular situation relevant to the second 
conclusion of Theorem \ref{theorem1.3}.

\begin{corollary}\label{corollary4.4}
Suppose that $k$ and $r$ are natural numbers with $r\ge k\ge 2$, and put $k_i=k+r(i-1)$ 
$(1\le i\le s)$. Then, provided that $s\ge (6k+6)^{2r}$, there is a positive number $\del$ 
having the property that
\[
\int_\grk |\mathscr F(\alp)|\d\alp \ll n^{\tet-1}L^{-\del}.
\]
\end{corollary}

\begin{proof} We apply Lemma \ref{lemma4.3}, observing that by a familiar argument one 
has
\[
\sum_{i=2}^s\frac{1}{k+r(i-1)}>\int_1^s\frac{\d t}{k+rt}=\frac{1}{r}\log 
\biggl( \frac{k+rs}{k+r}\biggr) .
\]
Thus, provided that one has
\begin{equation}\label{4.10}
\frac{1}{r}\log \biggl( \frac{k+rs}{k+r}\biggr) \ge 2\log k+\frac{2}{k+r}+1+\log (2D),
\end{equation}
then the conclusion of the corollary follows from Lemma \ref{lemma4.3}. We observe that 
the hypothesis $r\ge k$ ensures that
\[
\frac{1}{k+r}\le \frac{1}{2k}<\frac{1}{k}-\frac{1}{2k^2}<\log \Bigl( 1+\frac{1}{k}\Bigr),
\]
and hence $e^{1/(k+r)}<1+1/k$. It follows that the lower bound \eqref{4.10} is satisfied 
provided that
\[
k+rs\ge (2eD(k+1)^2)^r(k+r).
\]
Thus, on recalling that $k\le r$, we conclude that the lower bound \eqref{4.10} holds 
whenever $s\ge 2(2eD)^r(k+1)^{2r}$. In view of \eqref{3.12}, however, a modicum of 
computation confirms the bound $2^{1+1/r}eD<35$ for $r\ge 2$, and hence the upper 
bound asserted in the corollary holds whenever $s\ge (6k+6)^{2r}$, completing the proof. 
\end{proof}

\section{An enhanced minor arc estimate} Given exponents $k_i$ $(i\ge 1)$ having the 
property that $k_i$ is small for numerous small indices $i$, one may sharpen the analysis 
of the minor arcs presented in the previous section. We illustrate the underlying strategy 
in this section with a consideration of the situation in which
\[
k_i=k+r(i-1)\quad (1\le i\le s),
\]
with $r$ small. We now put
\begin{equation}\label{5.1}
F_2(\alp)=f_1(\alp)f_2(\alp)\cdots f_k(\alp)\quad \text{and}\quad 
G_2(\alp)=f_{k+1}(\alp)f_{k+2}(\alp)\cdots f_s(\alp).
\end{equation}
In accord with the discussion of the previous section, we have 
$F_2(0)G_2(0)\asymp n^\tet$. In all other respects, we adopt the notation of the previous 
section, with an analysis so similar that we are able to economise on detail.

\begin{lemma}\label{lemma5.1}
When $1\le Q\le n^{1/2}$, one has the bound
\[
\sup_{\alp \in \grm(Q)}|F_2(\alp)|\ll F_2(0)Q^{-1/(Dk(r+1))}.
\]
\end{lemma}

\begin{proof} As a consequence of Lemma \ref{lemma3.4}, it follows from \eqref{5.1} that
\begin{equation}\label{5.2}
\sup_{\alp\in \grm(Q)}|F_2(\alp)|\ll \prod_{i=1}^k P_iQ^{-1/(Dk_i^2)}=F_2(0)Q^{-\phi},
\end{equation}
in which we put
\[
\phi=\sum_{i=1}^k\frac{1}{D(k+r(i-1))^2}.
\]
However, by applying a familiar lower bound, we find that
\[
\phi>\frac{1}{D}\int_0^k\frac{\d t}{(k+rt)^2}=\frac{1}{Dr}
\Bigl( \frac{1}{k}-\frac{1}{k(r+1)}\Bigr) =\frac{1}{Dk(r+1)}.
\]
The conclusion of the lemma follows by substituting this estimate into \eqref{5.2}.
\end{proof}
 
Our upper bound for the mean value of $|G_2(\alp)|$ over $\grN(Q)$ is obtained through 
a modification of the corresponding treatment of $G_1(\alp)$ in Lemma \ref{lemma4.2}. We 
now put
\begin{equation}\label{5.3}
\Phi_2=\sum_{i=k+1}^s\frac{1}{k_i}.
\end{equation}

\begin{lemma}\label{lemma5.2} When $1\le Q\le n^{1/2}$, one has the bound
\[
\int_{\grN(Q)}|G_2(\alp)|\d\alp \ll G_2(0)n^{-1}Q^{2\Tet_2},
\]
where $\Tet_2=e^{1-\Phi_2+2/(k(r+1))}$.
\end{lemma}

\begin{proof} Define the exponents $t_i=k_i\Phi_2$ $(k+1\le i\le s)$. Then by following the 
argument of the proof of Lemma \ref{lemma4.2} mutatis mutandis, we obtain the upper 
bound
\[
\int_{\grN(Q)}|G_2(\alp)|\d\alp \ll \prod_{i=k+1}^s \left( P_i^{t_i}n^{-1}Q^{2\del_i}
\right)^{1/t_i},
\]
where we now write $\del_i=e^{1-\Phi_2+2/k_i}$. Thus we infer that
\[
\int_{\grN(Q)}|G_2(\alp)|\d\alp \ll G_2(0)n^{-1}Q^{2\Ome_2},
\]
where $\Ome_2=e^{1-\Phi_2+2/(k(r+1))}=\Tet_2$, and this delivers the conclusion of the 
lemma.
\end{proof}

We now combine the conclusions of Lemmata \ref{lemma5.1} and \ref{lemma5.2} much as 
in the proof of Lemma \ref{lemma4.3}.

\begin{lemma}\label{lemma5.3}
Suppose that $k$ and $r$ are natural numbers with $r\ge 1$ and $k\ge 2$, and put 
$k_i=k+r(i-1)$ $(1\le i\le s)$. Then, provided that $s\ge A(r)(k+1)^{r+1}$, where 
$A(r)=r^{-1}25^r(r+1)^{r+1}$, there is a positive number $\del$ having the property that
\[
\int_\grk |\mathscr F(\alp)|\d\alp \ll n^{\tet-1}L^{-\del}.
\]
\end{lemma}

\begin{proof} We again have $\grk\subset \grm(L^{1/15})$. When 
$L^{1/15}\le Q\le n^{1/2}$, we find from Lemmata \ref{lemma5.1} and \ref{lemma5.2} 
that
\begin{align}
\int_{\grN(Q)}|F_2(\alp)G_2(\alp)|\d\alp &\le \Bigl( \sup_{\alp \in \grN(Q)}|F_2(\alp)|
\Bigr) \int_{\grN(Q)}|G_2(\alp)|\d\alp \notag\\
&\ll F_2(0)Q^{-1/(Dk(r+1))}G_2(0)n^{-1}Q^{2\Tet_2}.\label{5.4}
\end{align}
On recalling \eqref{5.3}, we see that
\[
\Phi_2=\sum_{i=k+1}^s\frac{1}{k+r(i-1)}>\int_k^s\frac{\d t}{k+rt}=\frac{1}{r}\log 
\biggl( \frac{k+rs}{k(r+1)}\biggr) ,
\]
and hence
\[
\Tet_2<e^{1+2/(k(r+1))}\biggl( \frac{k(r+1)}{k+rs}\biggr)^{1/r}.
\]
Therefore, provided that
\begin{equation}\label{5.5}
2e^{1+2/(k(r+1))}\biggl( \frac{k(r+1)}{k+rs}\biggr)^{1/r}<\frac{1}{Dk(r+1)},
\end{equation}
then it follows from \eqref{5.4} that there is a positive number $\del$ having the property 
that
\begin{equation}\label{5.6}
\int_{\grN(Q)}|F_2(\alp)G_2(\alp)|\d\alp \ll F_2(0)G_2(0)n^{-1}Q^{-15\del}.
\end{equation}

\par We now set about establishing the inequality \eqref{5.5}. Observe first that since 
$r\ge 1$, one has
\[
\frac{2r}{k(r+1)^2}\le \frac{1}{2k}<\frac{1}{k}-\frac{1}{2k^2}<
\log \Bigl( 1+\frac{1}{k}\Bigr) ,
\]
and thus
\[
e^{2r/(k(r+1))}<(1+1/k)^{r+1}.
\]
We consequently infer that \eqref{5.5} holds whenever
\[
k+rs>k(r+1)(2eDk(r+1))^r (1+1/k)^{r+1}.
\]
On recalling \eqref{3.12}, we find that $2eD<25$, and thus it follows that \eqref{5.5} 
holds whenever
\[
s\ge r^{-1}25^r(r+1)^{r+1}(k+1)^{r+1}=A(r)(k+1)^{r+1}.
\]
Since this lower bound is one of the hypotheses of the lemma, we may henceforth work 
under the assumption that the upper bound \eqref{5.6} holds.\par

On recalling \eqref{2.5} and \eqref{2.1}, we find that \eqref{5.6} yields the bound
\[
\int_{\grN(Q)}|\mathscr F(\alp)|\d\alp \ll n^{\tet-1}Q^{-15\del}.
\]
A comparison with the concluding part of the argument of the proof of Lemma 
\ref{lemma4.3}, using a dyadic dissection of $\grk$ into subsets of the shape $\grN(Q)$, 
therefore leads us from here to the conclusion of the lemma. 
\end{proof}

\section{The major arc contribution} The goal of this section is to make progress on 
establishing the lower bound \eqref{2.8} for the contribution of the major arcs to 
$T(n;\eta)$. Throughout this section and the next, we work under the hypothesis that
\begin{equation}\label{6.1}
\tet=\sum_{i=1}^s\frac{1}{k_i}>2.
\end{equation}
The hypotheses available to us in Theorem \ref{theorem1.1} ensure that $\tet>3$, 
thereby confirming \eqref{6.1} with room to spare. In the first conclusion of Theorem 
\ref{theorem1.3}, meanwhile, we have in particular $s>25^r(k+1)$, and thus
\begin{align*}
\tet&=\sum_{i=1}^s\frac{1}{k+r(i-1)}>\int_0^s\frac{\d t}{k+rt}=\frac{1}{r}\log 
\biggl( \frac{k+rs}{k}\biggr) \\
&>\frac{1}{r}\log (25^r)=\log 25>2,
\end{align*}
and the hypothesis \eqref{6.1} again holds. On the other hand, in the second conclusion 
of Theorem \ref{theorem1.3} one has $r\ge k$ and $s\ge (6k+6)^{2r}$, whence a similar 
argument yields
\[
\tet>\frac{1}{r}\log \left( (6k+6)^r\right)=\log (6k+6)\ge \log 18>2,
\]
and \eqref{6.1} holds once again. The upshot of this discussion is that we are cleared in all 
circumstances to work henceforth under the assumption that \eqref{6.1} holds.\par

Suppose next that $\alp\in \grK(q,a)\subseteq \grK$. The standard theory of smooth Weyl 
sums (see \cite[Lemma 5.4]{Vau1989}) shows that there is a positive number $c=c(\eta)$ 
such that for $1\le i\le s$, one has
\[
f_i(\alp)=cq^{-1}S_i(q,a)v_i(\alp-a/q)+O(P_iL^{-1/4}),
\]
wherein
\[
S_i(q,a)=\sum_{t=1}^qe(at^{k_i}/q)\quad \text{and}\quad v_i(\bet)=\frac{1}{k_i}
\sum_{m\le n}m^{-1+1/k_i}e(\bet m).
\]
Put
\begin{equation}\label{6.2}
\grJ(n;X)=\int_{-X/n}^{X/n}v_1(\bet)v_2(\bet)\cdots v_s(\bet)e(-\bet n)\d\bet .
\end{equation}
Also, write
\[
\grS(n;X)=\sum_{1\le q\le X}q^{-s}U_n(q),
\]
where
\begin{equation}\label{6.4}
U_n(q)=\sum_{\substack{a=1\\ (a,q)=1}}^q S_1(q,a)S_2(q,a)\cdots S_s(q,a)e(-na/q).
\end{equation}
Then since $\grK$ has measure $O(L^{1/5}n^{-1})$, we see that
\begin{equation}\label{6.5}
\int_\grK \mathscr F(\alp)e(-n\alp)\d\alp =c^s\grJ(n;L^{1/15})\grS(n;L^{1/15})
+O(P_1P_2\cdots P_sn^{-1}L^{-1/20}).
\end{equation}

\par We show in the next section that the singular series
\begin{equation}\label{6.6}
\grS(n)=\sum_{q=1}^\infty q^{-s}U_n(q).
\end{equation}
converges absolutely and uniformly for $n\in \dbN$, and moreover that $\grS(n)\gg 1$ 
whenever $s\ge 4k_1$ and the condition \eqref{6.1} holds. Moreover, under the latter 
condition we show further that there is a positive number $\del$ such that
\begin{equation}\label{6.7}
\grS(n)-\grS(n;X)\ll X^{-\del}.
\end{equation}

\par Next, on making use of the bound supplied by \cite[Lemma 2.8]{Vau1997}, one finds 
that
\[
v_i(\bet)\ll P_i(1+n\| \bet\|)^{-1/k_i}\quad (1\le i\le s).
\]
Hence, working under the hypothesis \eqref{6.1}, we deduce from \eqref{6.2} that there is 
a positive number $\del$ such that
\begin{equation}\label{6.8}
\grJ(n;X)=\grJ(n)+O(P_1P_2\cdots P_sn^{-1}X^{-\del}),
\end{equation}
where
\[
\grJ(n)=\int_{-1/2}^{1/2}v_1(\bet)v_2(\bet)\cdots v_s(\bet)e(-\bet n)\d\bet .
\]
A familiar approach paralleling that of \cite[Theorem 2.3]{Vau1997} shows that
\begin{equation}\label{6.9}
\grJ(n)=\frac{\Gam\bigl( 1+\frac{1}{k_1}\bigr)\Gam\bigl( 1+\frac{1}{k_2}\bigr) \cdots 
\Gam\bigl( 1+\frac{1}{k_s}\bigr) }{\Gam \bigl( \frac{1}{k_1}+\frac{1}{k_2}+\ldots 
+\frac{1}{k_s}\bigr) }n^{\tet-1}\left( 1+O(n^{-1/k_s})\right) .
\end{equation}
Thus, on combining \eqref{6.5} with \eqref{6.7}, \eqref{6.8} and \eqref{6.9}, we conclude 
that there is a positive number $\del$ for which
\begin{equation}\label{6.10}
\int_\grK \mathscr F(\alp)e(-n\alp)\d\alp =c^s\Gam(\tet)^{-1}\biggl( 
\prod_{i=1}^s\Gam\Bigl( 1+\frac{1}{k_i}\Bigr) \biggr) 
\grS(n)n^{\tet-1}+O(n^{\tet-1}L^{-\del}).
\end{equation}

\par Subject to our verification in the next section that the lower bound $\grS(n)\gg 1$ holds 
uniformly in $n$, we conclude from \eqref{6.10} that the lower bound \eqref{2.8} holds. In 
combination with the minor arc estimate \eqref{2.7}, available from Lemma \ref{lemma4.3} 
under the hypotheses of Theorem \ref{theorem1.1}, we conclude that
\begin{equation}\label{6.11}
T(n;\eta)=\int_\grK \mathscr F(\alp)e(-n\alp)\d\alp +\int_\grk \mathscr F(\alp)e(-n\alp)
\d\alp \gg n^{\tet-1}.
\end{equation}
This completes the proof of Theorem \ref{theorem1.1}. In order to establish Theorem 
\ref{theorem1.3}, we observe on the one hand that the upper bound \eqref{2.7} follows 
from Lemma \ref{lemma5.3} when $s\ge A(r)(k+1)^{r+1}$. Also, when $r\ge k$ and 
$s\ge (6k+6)^{2r}$, the upper bound \eqref{2.7} follows from Corollary \ref{corollary4.4}. 
Thus, in either case, we find as before that \eqref{6.11} follows in these respective 
situations, and thus the proof of Theorem \ref{theorem1.3} is now complete.

\section{The singular series} In this section we estimate the singular series, confirming 
\eqref{6.7} and the bounds $1\ll \grS(n)\ll 1$. Our argument parallels the analogous 
treatment of \cite{Sco1960}, though we introduce refinements en route. We continue 
working under the hypothesis \eqref{6.1} throughout.\par

First, from \eqref{6.4} and \cite[Theorem 4.2]{Vau1997}, we see that the bound
\[
q^{-s}U_n(q)\ll q^{1-1/k_1-1/k_2-\ldots -1/k_s}=q^{1-\tet}
\]
holds uniformly in $n$. Thus, in view of \eqref{6.1}, there is a positive number $\del$ for 
which $q^{-s}U_n(q)\ll q^{-1-\del}$. It follows that the singular series $\grS(n)$ defined in 
\eqref{6.6} converges absolutely and uniformly in $n$, and moreover one has the bound 
\eqref{6.7}. Next, by following the argument underlying the proof of 
\cite[Lemma 2.11]{Vau1997}, we 
see that $U_n(q)$ is a multiplicative function of $q$. In view of \eqref{6.6}, we may rewrite 
$\grS(n)$ in the form $\grS(n)=\prod_p\chi_p(n)$, where the product is over all prime 
numbers $p$, and
\begin{equation}\label{7.1}
\chi_p(n)=\sum_{\nu=0}^\infty p^{-s\nu}U_n(p^\nu).
\end{equation}
By orthogonality, this Euler factor is related to the number $M_n(p^\nu)$ of incongruent 
solutions of the congruence
\[
x_1^{k_1}+x_2^{k_2}+\ldots +x_s^{k_s}\equiv n\mmod{p^\nu},
\]
via the relation
\begin{equation}\label{7.2}
\chi_p(n)=\lim_{\nu\rightarrow \infty}p^{(1-s)\nu}M_n(p^\nu).
\end{equation}
A model for the necessary argument, which is standard, may be found in the discussion 
associated with \cite[Lemma 2.12]{Vau1997}. The limit \eqref{7.2} is seen to exist via the 
relation \eqref{7.1}. In particular, the quantity $\chi_p(n)$ is a non-negative number 
satisfying the relation $\chi_p(n)=1+O(p^{-1-\del})$.\par

We summarise our deliberations thus far in the form of a lemma.

\begin{lemma}\label{lemma7.1}
Suppose that \eqref{6.1} holds. Then the series \eqref{6.6} converges absolutely, and there 
exists a natural number $C$ with the property that for all integers $n$, one has
\[
\grS(n)\ge \tfrac{1}{2} \prod_{p\le C}\chi_p(n).
\]
\end{lemma}

We have yet to obtain a lower bound for $\chi_p(n)$ when $p\le C$, a matter to which we 
now attend. Put $D=(k_1,k_2,\ldots ,k_s)$, the greatest common divisor of 
$k_1,k_2,\ldots ,k_s$. Define the non-negative integer $\lam$ by means of the relation 
$p^\lam\| D$. Then we have $p^\lam|k_i$ for $1\le i\le s$, and there exists an index $j$ 
with $1\le j\le s$ for which $p^\lam\| k_j$. We show that for each integer $n$, there is a 
solution of the congruence
\begin{equation}\label{7.3}
x_1^{k_1}+x_2^{k_2}+\ldots +x_s^{k_s}\equiv n\mmod{p^{\lam+\tau}},
\end{equation}
with $\tau=1$ for odd $p$, and with $\tau=2$ for $p=2$, in each case with $(x_j,p)=1$. 

\par In order to establish this last assertion, suppose temporarily that there is an integer 
$n$ having the property that \eqref{7.3} has no solution with $(x_j,p)=1$. It then follows 
that the range of the left hand side of \eqref{7.3} modulo $p^{\lam+\tau}$, with 
$(x_j,p)=1$, has at most $p^{\lam+\tau}-1$ elements. In the first instance we assume 
that $p$ is odd. Then, the theory of power residues shows that the monomial $x^{k_j}$ 
takes $(p-1)/(p-1,k_j)$ values modulo $p^{\lam+1}$ as $x$ varies over 
$1\le x\le p^{\lam+1}$ with $(x,p)=1$. Furthermore, for any index $i$ we see that 
$y^{k_i}$ takes at least $1+(p-1)/(p-1,k_i)$ values modulo $p^{\lam+1}$ as $y$ 
varies over $1\le y\le p^{\lam+1}$. We now repeatedly apply the Cauchy-Davenport 
theorem (see \cite[Lemma 2.14]{Vau1997}), beginning with the values of $x_j^{k_j}$, and 
then adding in the remaining powers step-by-step. On recalling \eqref{6.1}, we find that 
with $p\nmid x_j$, the range of the left hand side of \eqref{7.3}, modulo $p^{\lam+1}$, 
contains a number of elements which is at least
\[
\sum_{i=1}^s\frac{p-1}{(p-1,k_ip^{-\lam})}\ge p^\lam (p-1)
\sum_{i=1}^s\frac{1}{k_i}>2p^\lam(p-1).
\]
This yields a contradiction, since $2p^\lam (p-1)\ge p^{\lam+1}$. Our claim concerning the 
solubility of the congruence \eqref{7.3} is consequently confirmed when $p$ is odd.\par

We next consider the situation with $p=2$, where $\tau=2$. For some index $j$ with 
$1\le j\le s$, one has $2^\lam\| k_j$. In \eqref{7.3} we take $x_j=1$. We can solve 
\eqref{7.3} with $x_i\in \{0,1\}$ ($1\le i\le s$ and $i\ne j$) provided that 
$s\ge 2^{\lam+2}$. However, we have $k_1\ge 2^\lam$, and hence the condition that 
$s\ge 4k_1$ suffices to confirm our claim concerning the solubility of the congruence 
\eqref{7.3} in the case that $p=2$.\par

A routine argument now bounds $M_n(p^\nu)$ from below. We observe that since 
$p^\lam\|k_j$, a number coprime to $p$ is a $k_j$-th power residue modulo 
$p^{\lam+\tau}$ if and only if it is a $k_j$-th power residue modulo $p^\nu$, for all 
$\nu\ge \lam+\tau$. Let $x_1,x_2,\ldots ,x_s$ be a solution of \eqref{7.3}, with 
$(x_j,p)=1$, and let $\nu$ be a natural number with $\nu\ge \lam+\tau$. There are 
$p^{\nu-\lam-\tau}$ choices for $y_i$ with $y_i\equiv x_i\mmod{p^{\lam+\tau}}$ and 
$1\le y_i\le p^\nu$. For each such choice with $1\le i\le s$ and $i\ne j$, the integer
\[
n-\sum_{\substack{i=1\\ i\ne j}}^sx_i^{k_i}
\]
is a $k_j$-th power residue modulo $p^{\lam+\tau}$, and therefore a $k_j$-th power 
residue modulo $p^\nu$. Thus, we have $M_n(p^\nu)\ge p^{(s-1)(\nu -\lam-\tau)}$, so by 
\eqref{7.2} we see that $\chi_p(n)\ge p^{-(\lam+\tau)(s-1)}$. This lower bound holds for 
all primes $p$ with $p^\lam\| D$ and all $n\in \dbN$ provided that $s\ge 4k_1$ and 
\eqref{6.1} holds.\par

We summarise these deliberations in the following lemma.

\begin{lemma}\label{lemma7.2}
Suppose that $s\ge 4k_1$, and \eqref{6.1} holds. Then there is a positive number $\ome$ 
having the property that $\grS(n)\ge \ome$ for all $n\in \dbN$. 
\end{lemma}

This lemma completes our analysis of the singular series, and thus we have confirmed all of 
the properties that were needed to complete the analysis of \S6. It is worth noting that the 
condition $s\ge 4k_1$ of Lemma \ref{lemma7.2} is automatically satisfied whenever the 
hypotheses of Theorem \ref{theorem1.1} hold for the exponents $k_1,k_2,\ldots ,k_s$. In 
order to verify this claim, observe that
\[
\sum_{i=3}^s\frac{1}{k_3}\le \frac{s}{k_1}
\]
whilst
\[
2\log k_1+\frac{1}{k_2}+3.20032>2\log 2+3>4.
\]
Thus the hypotheses of Theorem \ref{theorem1.1} can be satisfied only when $s>4k_1$.

\bibliographystyle{amsbracket}

\begin{thebibliography}{18}
\bibitem{Bru1987}
J. Br\"udern, \emph{Sums of squares and higher powers II}, J. London Math. Soc. (2) 
\textbf{35} (1987), no. 2, 244--250.

\bibitem{Bru1988}
J. Br\"udern, \emph{A problem in additive number theory}, Math. Proc. Cambridge Philos. 
Soc. \textbf{103} (1988), no. 1, 27--33.

\bibitem{BW2022}
J. Br\"udern and T. D. Wooley, \emph{On Waring's problem for larger powers}, submitted, 
arxiv:2211.10380.

\bibitem{For1995}
K. B. Ford, \emph{The representation of numbers as sums of unlike powers}, J. London 
Math. Soc. (2) \textbf{51} (1995), no. 1, 14--26.

\bibitem{For1996}
K. B. Ford, \emph{The representation of numbers as sums of unlike powers. II}, J. Amer. 
Math. Soc. \textbf{9} (1996), no. 4, 919--940.

\bibitem{Fre1949}
G. A. Fre\u \i man, \emph{Solution of Waring's problem in a new form}, Uspehi Matem. 
Nauk (N.S.) \textbf{4} (1949), no. 1 (29), 193.

\bibitem{KLX2020}
C.~I.~Kuan, D.~Lesesvre and X.~Xiao, \emph{Sums of even ascending powers}, submitted, 
arxiv:2001.02429.

\bibitem{LZ2021}
J. Liu and L. Zhao, \emph{Representation by sums of unlike powers}, J. Reine Angew. Math. 
\textbf{781} (2021), 19--55.

\bibitem{Mon1971}
H. L. Montgomery, \emph{Topics in multiplicative number theory}, Springer-Verlag, 
Berlin-New York, 1971.

\bibitem{Rot1951}
K. F. Roth, \emph{A problem in additive number theory}, Proc. London Math. Soc. (2) 
\textbf{53} (1951), no. 1, 381--395.

\bibitem{Sco1960}
E. J. Scourfield, \emph{A generalization of Waring's problem}, J. London Math. Soc. 
\textbf{35} (1960), no. 1, 98--116.

\bibitem{Tha1968}
K. Thanigasalam, \emph{On additive number theory}, Acta Arith. \textbf{13} (1967/68), 
no. 3, 237--258.

\bibitem{Tha1980}
K. Thanigasalam, \emph{On sums of powers and a related problem}, Acta Arith. 
\textbf{36} (1980), no. 2, 125--141.

\bibitem{Tha1982}
K. Thanigasalam, \emph{Addendum and corrigendum to ``On sums of powers and a 
related problem''}, Acta Arith. \textbf{42} (1982), no. 4, 425.

\bibitem{Tha1984}
K. Thanigasalam, \emph{On certain additive representations of integers}, Portugal. Math. 
\textbf{42} (1983/84), no. 4, 447--465.

\bibitem{Vau1970}
R. C. Vaughan, \emph{On the representation of numbers as sums of powers of natural 
numbers}, Proc. London Math. Soc. (3) \textbf{21} (1970), no. 1, 160--180.

\bibitem{Vau1971}
R. C. Vaughan, \emph{On sums of mixed powers}, J. London Math. Soc. (2) \textbf{3} 
(1971), no. 4, 677--688.

\bibitem{Vau1997}
R. C. Vaughan, \emph{The Hardy-Littlewood method}, 2nd edition, Cambridge University 
Press, Cambridge, 1997.

\bibitem{Vau1989}
R. C. Vaughan, \emph{A new iterative method in Waring's problem}, Acta Math. 
\textbf{162} (1989), no. 1-2, 1--71.

\bibitem{VW1991}
R. C. Vaughan and T. D.  Wooley, \emph{On Waring's problem: some refinements}, Proc. 
London Math. Soc. (3) \textbf{63} (1991), no. 1, 35--68.

\bibitem{Woo1993}
T. D. Wooley, \emph{The application of a new mean value theorem to the fractional parts of 
polynomials}, Acta Arith. \textbf{65} (1993), no. 2, 163--179.

\bibitem{Woo2015a}
T. D. Wooley, \emph{Sums of three cubes, II}, Acta Arith. \textbf{170} (2015), no. 1, 
73--100.

\bibitem{Woo2015b}
T. D. Wooley, \emph{Rational solutions of pairs of diagonal equations, one cubic and one 
quadratic}, Proc. London Math. Soc. (3) \textbf{110} (2015), no. 2, 325--356.

\end{thebibliography}
\providecommand{\bysame}{\leavevmode\hbox to3em{\hrulefill}\thinspace}

\end{document}